\documentclass{amsart}
\usepackage{amssymb}
\usepackage{graphicx}
\usepackage{amsmath}
\newtheorem{theorem}{Theorem}

\newtheorem{corollary}[theorem]{Corollary}

\newtheorem{example}[theorem]{Example}

\newtheorem{lemma}[theorem]{Lemma}

\newtheorem{proposition}[theorem]{Proposition}
\newtheorem{remark}[theorem]{Remark}

\begin{document}
\title{Shorted\ Operators Relative to a Partial Order in a Regular Ring}
\author{Brian Blackwood}
\address{Department of Mathematics, Potomac State College, West Virginia University, Keyser, West Virginia, USA}
\email{Brian.Blackwood@mail.wvu.edu}
\author{S. K. Jain}
\address{Department of Mathematics, Ohio University, Athens, Ohio-45701, USA}
\email{jain@math.ohiou.edu}
\author{K. M. Prasad}
\address{Manipal Institue of Management\\
Manipal University\\
Manipal 576 104, Karnataka, India}
\email{karantha.prasad@gmail.com}
\author{Ashish K. Srivastava}
\address{Department of Mathematics and Computer Science, St. Louis University, St.
Louis, MO-63103, USA}
\email{asrivas3@slu.edu}

\begin{abstract}
In this paper, the explicit form of maximal elements, known as shorted operators, in a subring of a von Neumann regular ring has been obtained. As an application of the main theorem, the unique shorted operator (of electrical circuits) which was introduced by Anderson-Trapp has been derived.
\end{abstract}

\maketitle

\section{\protect\bigskip Introduction}

Various partial orders on an abstract ring or on the ring of matrices over
the real and complex numbers have been introduced by several authors either
as an abstract study of questions in algebra, or for the study of problems
in engineering and statistics (See, e.g. \cite{A1}, \cite{A2}, \cite{D1}, 
\cite{J4}, \cite{L4}, and \cite{M3}). \ Also, a partial order on semigroups
is studied by several authors (See, e.g. \cite{H1}, \cite{M6}, and \cite{N1}%
). \ In this paper we study the well-known minus partial order on a von
Neumann regular ring which is simply a generalization of a partial order on
the set of idempotents in a ring introduced by Kaplansky. \ For any two
elements $a,b$ in a von Neumann regular ring $R$, we say $a\leq ^{-}b$ (and
read it as $a$ is less than or equal to $b$ under the \textit{minus partial
order)} if$\ $there exists an $x\in R$ such that $ax=bx$ and $xa=xb$ where $%
axa=a$. \ Furthermore, we define the partial order $\leq ^{\oplus }$ by
saying that $a\leq ^{\oplus }b$ if $bR=aR\oplus (b-a)R,$ and call it the 
\textit{direct sum partial order}. \ The \textit{Loewner partial order} on
the set of positive semidefinite matrices $S$ is defined by saying that for $%
a,b\in S$, $a\leq _{L}b$ if $b-a\in S$. \ The direct sum partial order is
shown to be equivalent to the minus partial order on a von Neumann regular
ring. It is known that the minus partial order on the subset of positive
semidefinite matrices in the matrix ring over the field of complex numbers
implies the Loewner partial order.\ \ The main result of this paper gives an
explicit description of maximal elements in a subring under minus partial
order (Theorem 13). \ As a special case, we obtain a result similar to the
one obtained by Mitra-Puri (\cite{M3}, Theorem $2.1$) for the unique shorted operator; which,
in turn, is equivalent to the formula of Anderson-Trapp (\cite{A2}, Theorem
$1$) for computing the shorted operator of a shorted electrical circuit
(Theorem 17).

\bigskip

\section{Definitions}

Throughout this paper, $R$ is a ring with identity. An element $a\in R$ is
called von Neumann regular if $axa=a$ for some $x\in R$ and $x$ is called a
von Neumann inverse of $a$. We will denote an arbitrary von Neumann inverse
of $a$ by $a^{(1)}$. An element $a\in R$ is called weakly regular if $xax=x$
for some $x\in R$ and $x$ is called a weak von Neumann inverse of $a$. We
will denote a weak von Neumann inverse of $a$ by $a^{(2)}$. If $axa=a$ and $%
xax=x$, then $x$ is called a strong von Neumann inverse of $a$. We will
denote a strong von Neumann inverse of $a$ by $a^{(1,2)}$. A ring $R$ is
called von Neumann regular if every element in $R$ is von Neumann regular.
For convenience, we will use the terminology \textit{regular ring} in place
of \textit{von Neumann regular ring}. For details on regular ring, the reader is referred to \cite {G1}.

Let $S$ be the set of all regular elements in any ring $R$. \ For $a,b\in S$
we say that $a\leq ^{-}b$ if there exists a von Neumann inverse $x$ of $a$
such that $ax=bx$ and $xa=xb$. \ This is known as the minus partial order as
stated above for regular rings. The minus partial order clearly generalizes
the definition of Kaplansky according to which if $e,f$ are idempotents then 
$e\leq f$ if $ef=e=fe.$\ 

We remark that for the ring of matrices over a field, it is known that $%
a\leq ^{-}b$ if and only if $rank(b-a)=rank(b)-rank(a)$.

Let $T$ be a ring with involution *. If $x$ is a strong von Neumann inverse
of $a$ such that $(ax)^{\ast }=ax$, $(xa)^{\ast }=xa$ and $ax=xa$ then $x$
is called the \textit{Moore-Penrose inverse} of $a$ and is denoted by $%
a^{\dag }$. \ Let $M$ be the set of positive semidefinite matrices. For $%
w\in M$ and $b\in T$, $x$ is called the unique $w$\textit{-weighted
Moore-Penrose inverse} of $b$ if $x$ is a strong von Neumann inverse of $b$
and satisfies $(wbx)^{\ast }=wbx$ and $(wxb)^{\ast }=wxb$. For details on
Moore-Penrose inverse, one may refer to Rao-Mitra \cite{R1} or Ben-Israel and
Greville \cite{B2}.\ \ 

\bigskip

\section{\protect\bigskip Preliminary Results}

The following result of Jain and Prasad (\cite{J5}, Theorem 1) will prove to be useful
throughout this paper and, specifically, for providing an equivalent
definition of the minus partial order on a regular ring.

\begin{theorem}
\label{p1}Let R be a ring and let $a,b\in R$ such that $a+b$ is a regular
element. \ Then the following are equivalent:

\begin{enumerate}
\item  $aR\oplus bR=(a+b)R;$

\item  $Ra\oplus Rb=R(a+b);$

\item  $aR\cap bR=(0)=Ra\cap Rb.$
\end{enumerate}
\end{theorem}

\bigskip

From Rao-Mitra (\cite{R1}, Theorem 2.4.1, page 26), we have the following nice characterization of $%
\{a^{(1)}\}$ and $\{a^{(1,2)}\}$.

\begin{lemma}
\label{p2}Let R be a ring and let $a\in R$. \ If $x\in \{a^{(1)}\}$ then $%
\{a^{(1)}\}=x+(1-xa)R+R(1-ax).$ \ In addition, $\{a^{(1,2)}\}=%
\{a^{(1)}aa^{(1)}\}$.
\end{lemma}

We now investigate properties of the direct sum partial order and its
relation to the minus partial order. \ 

Let $R$ be a regular ring. \ Recall $a\leq ^{\oplus }b$ if and only if $%
bR=aR\oplus (b-a)R$. \ By Theorem 1, this is equivalent to $Rb=Ra\oplus
R(b-a)$.\ \ It is straightforward to see that $\leq ^{\oplus }$ is a partial
order.

Next we show that the minus partial order is equivalent to the direct sum
partial order on a regular ring. \ Hartwig-Luh showed that, when $%
R $ is a regular ring, $(2)$ is equivalent to $(3)$ with the additional
hypothesis that $a\in bRb$ (see \cite{M4}, page $5$).

\begin{lemma}
\label{p3}Let $R$ be a regular ring and $a,b\in R$. \ Then the following are
equivalent:
\end{lemma}

\begin{enumerate}
\item  $a\leq ^{\oplus }b;$

\item  $a\leq ^{-}b;$

\item  $\{b^{(1)}\}\subseteq \{a^{(1)}\}.$
\end{enumerate}

\begin{proof}
$(1)\Longrightarrow (2):$ \ As $a\leq ^{\oplus }b$, $bR=aR\oplus (b-a)R$. \ It
follows that $aR\subseteq bR$. $\ $Hence, $a\in bR$ and thus $a=bx$ for some 
$x\in R$. \ As $R$ is a regular ring, for any $g\in \{b^{(1)}\}$, $bgb=b$. \
Thus $bga=bg(bx)=(bgb)x=bx=a$. \ Now $aga=bga-(b-a)ga=a-(b-a)ga$. \ Thus $%
a-aga=(b-a)ga$. \ But $aR\cap (b-a)R=(0)$ and $a-aga=(b-a)ga\in aR\cap
(b-a)R $. \ Hence $a-aga=0$ and $(b-a)ga=0$. \ Therefore $aga=a=bga$ and
hence $\{b^{(1)}\}\subseteq \{a^{(1)}\}$. \ Indeed, this demonstrates that $%
(1)\Longrightarrow (3).$ \ Now choose $x=gag$. \ Then $axa=a(gag)a=aga=a$ and $%
x\in \{a^{(1)}\}$. \ Now $bx=(bga)g=ag$ as $bga=a$. \ Furthermore, $%
ax=agag=ag$ as $aga=a$. \ Thus $ax=bx$. \ Now $bg(b-a)=bgb-bga=(b-a)$ and $%
(b-a)g(b-a)=bg(b-a)-ag(b-a)=(b-a)-ag(b-a)$. \ Hence $%
ag(b-a)=(b-a)-(b-a)g(b-a)\in aR\cap (b-a)R=(0)$. \ Thus $(b-a)=(b-a)g(b-a)$
and $ag(b-a)=0$. \ It follows that $agb=aga=a$. \ Now $xb=(gag)b=g(agb)=ga$
and $xa=gaga=ga$. \ Therefore $xb=xa$. \ Thus $ax=bx$ and $xa=xb$ for some $%
x\in \{a^{(1)}\}$ and it follows that $a\leq ^{-}b$.

$(2)\Longrightarrow (3):$ This is well-known. We prove it here for completeness.
As $a\leq ^{-}b$, there exists some $x\in \{a^{(1)}\}$ such that $ax=bx$ and 
$xa=xb$. \ It follows that $a=axa=bxa=axb$ and for any $y\in \{b^{(1)}\}$, $%
aya=(axb)y(bxa)=ax(byb)xa=axbxa=(axb)xa=axa=a$. \ Thus $\{b^{(1)}\}\subseteq
\{a^{(1)}\}$.

$(3)\Longrightarrow (1):$ \ Given that $\{b^{(1)}\}\subseteq \{a^{(1)}\}$, $%
ab^{(1)}a=a$ for any $b^{(1)}\in \{b^{(1)}\}$. \ By Lemma \ref{p2}, $%
\{b^{(1)}\}=g+(1-gb)R+R(1-bg)$ for $g\in \{b^{(1)}\}$. \ For each $x\in
\{b^{(1)}\}$ there exists some $r_{1},r_{2}\in R$ such that $%
x=g+(1-gb)r_{1}+r_{2}(1-bg)$. \ Multiplying on the left and right by $a$
yields $axa=a\left[ g+(1-gb)r_{1}+r_{2}(1-bg)\right] a$. \ Hence $a=axa=a%
\left[ g+(1-gb)r_{1}+r_{2}(1-bg)\right]
a=aga+a(1-gb)r_{1}a+ar_{2}(1-bg)a=a+a(1-gb)r_{1}a+ar_{2}(1-bg)a$. \ Thus $%
a(1-gb)r_{1}a+ar_{2}(1-bg)a=0.$ \ As $a(1-gb)r_{1}a+ar_{2}(1-bg)a=0$ holds
for all $r_{1}$ and $r_{2}$, we can take, in particular, $r_{2}=0$ which
gives $a(1-gb)r_{1}a=0$ for all $r_{1}$ and hence $a(1-gb)Ra=(0)$.
Similarly, by taking $r_{1}=0$, we conclude $aR(1-bg)a=(0)$. 
Now $\left( a(1-gb)R\right) ^{2}=\left( a(1-gb)R\right) \left(
a(1-gb)R\right) =\left( a(1-gb)Ra\right) \left( (1-gb)R\right) =(0)\left(
(1-gb)R\right) =(0)$. \ Similarly $\left( R(1-bg)a\right) ^{2}=(0)$. \ Since 
$R$ is a regular ring, it has no nonzero nilpotent left or right ideal. \
Thus, $a(1-gb)R=(0)$ and $R(1-bg)a=(0)$. \ As $1\in R$, $a(1-gb)=0$ and $%
(1-bg)a=0$. \ Therefore, $bga=a=agb$.\ Now for any $t_{1},t_{2}\in R$, $%
at_{1}=(bga)t_{1}=b(gat_{1})\in bR$ and $(b-a)t_{2}=bt_{2}-at_{2}=bt_{2}-%
\left( bga\right) t_{2}=b(t_{2}-gat_{2})\in bR$. \ Hence, $%
aR+(b-a)R\subseteq bR$ . \ Thus $aR+(b-a)R=bR$. \ Now we want to show that $%
aR\cap (b-a)R=(0)$. \ For some $u,v\in R$, suppose $au=(b-a)v\in aR\cap
(b-a)R$. \ Then \ $au=agau=ag(b-a)v=agbv-agav=av-av=0$ as $a=agb$. \ Thus $%
aR\cap (b-a)R=(0)$ and so $bR=aR\oplus (b-a)R$. \ Hence, $a\leq ^{\oplus }b$
as required.
\end{proof}

\ We also note that proving directly $(2)\implies (1)$ requires a brief argument.

\bigskip

\bigskip

The Corollary that follows shows, in particular, that the minus partial
order defined on the set of idempotents is the same as the partial order
defined by Kaplansky on idempotents (See e.g. Lam \cite{L1}, page 323).

\begin{corollary}
\label{p4}Let $R$ be a regular ring and $a,b\in R$ such that $b=b^{2}$. \
Then the following are equivalent:

\begin{enumerate}
\item  $a\leq ^{-}b;$

\item  $a=a^{2}=ab=ba.$
\end{enumerate}
\end{corollary}

\begin{proof}
The proof is straightforward.
\end{proof}

\begin{corollary}
\label{p9}Let $R$ be a regular ring and let $a,b,c\in R$ with $b=a+c$. \
Then the following statements are equivalent:

\begin{enumerate}
\item  $a\leq ^{-}b$;

\item  $aR\cap cR=(0)=Ra\cap Rc$.
\end{enumerate}
\end{corollary}

\begin{proof}
It follows from Lemma 3 and observing that, in a regular ring, $a\leq ^{-}a+c
$ if and only if$\ a\leq ^{\oplus }a+c$ if and only if $(a+c)R=aR\oplus cR$. 
\end{proof}

\bigskip

Hartwig (\cite{H1}, Pages 12-13) posed the following questions, among others:

\bigskip

(1) If $R$ is a regular ring and $aR\cap cR=(0)=Ra\cap Rc$, does there exist 
$a^{(1)}$ such that $a^{(1)}c=0=ca^{(1)}$?\ 

\bigskip

(2) \ Does $a\leq ^{-}c$, $b\leq ^{-}c$, $aR\cap cR=(0)=Ra\cap Rc$ imply $%
a+b\leq ^{-}c$?

\bigskip

\ As a byproduct of the development of the direct sum partial order, we give
an application that answers the above two questions of Hartwig. We do not
know whether or not someone has answered these questions, as we could not
find this in the literature. \ In any case, we believe that the answers we
have given would be of interest to the reader. \ Below, we answer Question 1
in the affirmative and Question 2 in the negative by providing a
counterexample. \ 

\begin{proposition}
\label{p10}$($Hartwig Question $1$$)$ If $R$ is a regular ring and $aR\cap
cR=(0)=Ra\cap Rc$, for some nonzero elements $a,c\in R$, then there exists a
nonzero $a^{(1)}$ such that $a^{(1)}c=0=ca^{(1)}$.
\end{proposition}

\begin{proof}
Let $b=a+c$.\ \ By Corollary 5, $a\leq ^{-}b$.\ \ Then, by the definition of
\ the minus partial order, for some $a^{(1)}$, $aa^{(1)}=ba^{(1)}$ and $%
a^{(1)}a=a^{(1)}b$.\ \ Now substituting $b=a+c$ yields $aa^{(1)}=(a+c)a^{(1)}
$ and $a^{(1)}a=a^{(1)}(a+c)$.\ \ Thus $aa^{(1)}=aa^{(1)}+ca^{(1)}$ and $%
a^{(1)}a=a^{(1)}a+a^{(1)}c$.\ \ It follows that $ca^{(1)}=0=a^{(1)}c$ as
required.
\end{proof}

\begin{example}
\label{p11}$($Hartwig Question $2$$)$

Using matrix units $e_{ij}$, let $a=e_{13}$, $b=e_{24}$, and $%
c=e_{13}+e_{14}+e_{24} $. \ Clearly $a\leq ^{-}c$ and $b\leq ^{-}c.$ \ It is
obvious that $aR\cap bR=(0)=Ra\cap Rb$. \ Since $rank(c)-rank(a+b)=2-2=0$
and $rank(c-(a+b))=1$, it follows that $a+b\nleq ^{-}c$. \ 
\end{example}

\bigskip

\section{Main Results}

Let $R$ be a regular ring and $S$ be a subset of $R$. We define a maximal
element in $C=\{x\in S:x\leq ^{\oplus }a\}$ as an element $b\neq a$ such
that $b$ $\leq ^{\oplus }a$ and if $b$ $\leq ^{\oplus }c\leq ^{\oplus }a$
then $c=b$ or $c=a$.

For fixed elements $a,b,c\in R,$ we give a complete description of the
maximal elements in the subring $S=eRf$, where $e$ and $f$ are idempotents
given by $eR=aR\cap cR$ and $Rf=Ra\cap Rb$. \ Here, $C=\{s\in eRf:s\leq
^{\oplus }a\}$. \ In the literature, maximal elements in $C$ have been
called shorted operators of $a$ (\cite{A1}, \cite{A2} and \cite{M3}).

We begin with a result that is used frequently in the sequel. This is indeed
contained in (\cite{M6}, Lemma $1$) where the author proves the equivalence of
$11$ statements. \ However, for the sake of completeness, we provide a direct
argument.

\begin{lemma}
\label{p12}Suppose $R$ is a regular ring and $a,b\in R$ such that $%
\{a^{(1)}\}\cap \{b^{(1)}\}\neq \emptyset $. \ Then the following are
equivalent:

\begin{enumerate}
\item  $aR\subset bR$ and $Ra\subset Rb;$

\item  $a\leq ^{\oplus }b.$
\end{enumerate}
\end{lemma}

\begin{proof}
Suppose $aR\subset bR$ and $Ra\subset Rb$. \ It follows that $a=rb=bs$ for
some $r,s\in R$. \ We claim that $ab^{(1)}a$ is invariant under any choice
of $b^{(1)}$. \ Let $\ x,y\in \{b^{(1)}\}$ be arbitrary. \ Now $%
axa=(rb)x(bs)=r(bxb)s=rbs$ as $bxb=b$. \ Similarly, $%
aya=(rb)y(bs)=r(byb)s=rbs$ as $byb=b$. \ Thus $axa=aya$ for every $x,y\in
\{b^{(1)}\}$. \ Hence $ab^{(1)}a$ is invariant under any choice of $b^{(1)}$%
. Since we have assumed that $\{a^{(1)}\}\cap \{b^{(1)}\}\neq \emptyset $, there
exists some $g\in \{a^{(1)}\}\cap \{b^{(1)}\}$. \ Therefore $ab^{(1)}a=aga=a$
for all $b^{(1)}$. \ Hence $\{b^{(1)}\}\subseteq \{a^{(1)}\}$ and by Lemma 
\ref{p3}, $a\leq ^{\oplus }b$.

Conversely, if $a\leq ^{\oplus }b$, then $aR\subset bR$ and $Ra\subset Rb$
follow by definition.
\end{proof}

We now demonstrate an important relationship between weak von Neumann
inverses and strong von Neumann inverses under the direct sum partial order.

\begin{lemma}
\label{p13}Let $a\in R$ where $R$ is a regular ring. \ Then the following
are equivalent:

\begin{enumerate}
\item  $b$ is a weak von Neumann inverse of $a;$

\item  There exists a strong von Neumann inverse $c$ of $a$ such that $b\leq
^{\oplus }c.$
\end{enumerate}
\end{lemma}

\begin{proof}
Suppose $b$ is a weak von Neumann inverse of $a$. \ For any
fixed $a^{(1)}$, define $u=a^{(1)}(a-aba)a^{(1)}$ and $c=b+u$. \ Then $%
aca=aba+aua=aba+aa^{(1)}aa^{(1)}a-aa^{(1)}abaa^{(1)}a=aba+a-aba=a$ and $%
cac=(b+u)a(b+u)=bab+bau+uab+uau=b+ba(a^{(1)}aa^{(1)}-a^{(1)}abaa^{(1)})+(a^{(1)}aa^{(1)}-a^{(1)}abaa^{(1)})ab+
$

$%
(a^{(1)}aa^{(1)}-a^{(1)}abaa^{(1)})a(a^{(1)}aa^{(1)}-a^{(1)}abaa^{(1)})=b+baa^{(1)}-baa^{(1)}+a^{(1)}ab-a^{(1)}ab+a^{(1)}aa^{(1)}-a^{(1)}abaa^{(1)}-a^{(1)}abaa^{(1)}+a^{(1)}abaa^{(1)}=b+a^{(1)}(a-aba)a^{(1)}=b+u=c
$. \ This shows that $c$ is a strong von Neumann inverse of $a$. \ 

Now we want to show that $b\leq ^{\oplus }c$. \ In other words, we will
prove that $bR\oplus uR=cR$. \ Observe that $%
cab=[b+a^{(1)}(a-aba)a^{(1)}]ab=bab+a^{(1)}(ab-abab)=bab=b$. \ Therefore $b\in
cR$. \ As $c=b+u$, it is clear that $cR\subseteq bR+uR$. \ As $u=c-b$ and $%
b\in cR$, $uR\subseteq $ $cR$. \ It follows that $cR=bR+uR$. \ Now we want
to show that $bR\cap uR=(0)$. \ Let $bp=uq\in bR\cap uR$ for some $p,q\in R$%
. \ Multiplying $ba$ on both sides yields $%
bp=babp=bauq=ba[a^{(1)}(a-aba)a^{(1)}]q=(ba-baba)a^{(1)}q=(ba-ba)a^{(1)}q=0$%
. \ Therefore $bR\cap uR=0$. \ Thus $bR\oplus uR=cR$ and we have
demonstrated that $b\leq ^{\oplus }c$.

Conversely, suppose that there exists a strong von Neumann inverse $c$ of $a$
such that $b\leq ^{\oplus }c$. As $c$ is a weak von Neumann inverse of $a$, $%
cac=c$ and thus $a\in \{c^{(1)}\}$. \ By assumption $b\leq ^{\oplus }c$ and
it follows from Lemma \ref{p3} that $\{c^{(1)}\}\subseteq \{b^{(1)}\}$. \
Thus $a\in \{c^{(1)}\}\subseteq \{b^{(1)}\}$ and it follows that $bab=b$. \
Hence $b$ is a weak von Neumann inverse of $a$.
\end{proof}

\begin{lemma}
\label{p14}Suppose $R$ is a regular ring. \ Let $y$ be a weak von Neumann
inverse and $z$ be a strong von Neumann inverse of an element $\alpha $ in
the subring $fRe$\ such that $y\leq ^{\oplus }z$. \ Then $eyf\leq ^{\oplus
}ezf$.

\begin{proof}
Let $\alpha =fxe\in fRe$. \ Since $y\leq ^{\oplus }z$, $yR\subseteq zR$ and $%
Ry\subseteq Rz$. \ Thus, $y=rz=zs$ for some $r,s\in R$. \ It is
straightforward to verify that $z\alpha y=y=y\alpha z$. \ This gives $%
(ezf)x(eyf)=(ezf)x(e(zs)f)=ez(fxe)zsf=ezsf=eyf$. \ Similarly $(eyf)x(ezf)=eyf
$. \ Thus $(eyf)R\subseteq (ezf)R$ and $R(eyf)\subseteq R(ezf)$. \ As $%
\alpha =fxe$ is a common von Neumann inverse of $y$ and $z$, it follows that 
$(eyf)x(eyf)=eyf$ and $(ezf)x(ezf)=ezf$ and so $x$ is a common von Neumann
inverse of $eyf$ and $ezf$. \ By Lemma \ref{p12}, $eyf\leq ^{\oplus }ezf$ .
\end{proof}
\end{lemma}

\ Next, we give two key lemmas. We will assume throughout that $a\not\in S$.

\begin{lemma}
\label{p15}Let $R$ be a regular ring. \ Then $d\in C$ is a maximal element
in $C$ \ if and only if for any $d^{^{\prime }}\leq ^{\oplus }a$ such that $%
dR\subseteq d^{^{\prime }}R\subseteq eR$, $Rd\subseteq Rd^{^{\prime
}}\subseteq Rf$, we have $d=d^{^{\prime }}$.
\end{lemma}

\begin{proof}
Let $d$ be a maximal element in $C$. \ If $d^{^{\prime }}$ is any element in 
$R$ such that $d^{^{\prime }}\leq ^{\oplus }a$ and $dR\subseteq d^{^{\prime
}}R\subseteq eR$, $Rd\subseteq Rd^{^{\prime }}\subseteq Rf$, then clearly $%
d^{^{\prime }}\in eRf$. \ As $d^{^{\prime }}\leq ^{\oplus }a$, $d^{^{\prime
}}\in C$. \ Then $\{a^{(1)}\}\subseteq \{d^{(1)}\}\cap \{(d^{^{\prime
}})^{(1)}\}$.\ \ Hence, $d\leq ^{\oplus }d^{^{\prime }}$ by Lemma \ref{p12}%
.\ \ Then by the maximality of $d$ in $C$, $d$ $=$ $d^{^{\prime }}$.

The converse is obvious.
\end{proof}

\begin{lemma} \label{p16}
$C=\{euf:u$ is a weak von Neumann inverse of $fa^{(1)}e\}$.
\end{lemma}

\begin{proof}
Let $s=etf\in C$ for some $t\in R$. Then $s\leq ^{\oplus }a$. \ By Lemma \ref
{p3}, $\{a^{(1)}\}\subseteq \{s^{(1)}\}$. \ Therefore, we have $%
(etf)a^{(1)}(etf)=(etf)$. In other words, $(etf)(fa^{(1)}e)(etf)=(etf)$,
proving that $s=etf$ is a weak von Neumann inverse of $fa^{(1)}e$. \ This
shows that $s=euf$ for some weak von Neumann inverse $u$\ of\ $fa^{(1)}e$.

Conversely, consider any $u\in (fa^{(1)}e)^{(2)}$ and let $x=euf$. \ We want
to show that $x\leq ^{\oplus }a$. \ Now $xa^{(1)}x=\left( euf\right)
a^{(1)}\left( euf\right) =eu\left( fa^{(1)}e\right) uf=euf=x$ as $u\in
(fa^{(1)}e)^{(2)}$. \ Hence $\{a^{(1)}\}\subseteq \{x^{(1)}\}$. \ By Lemma 
\ref{p3}, $x\leq ^{\oplus }a$ and so $x=euf\in C$. \ 
\end{proof}

\begin{theorem} \label{p17}
$\max C=\{evf:v$ is a strong von Neumann inverse of $fa^{(1)}e\}.$
\end{theorem}

\begin{proof}
Suppose $x=euf\in C$ where $u=\left( fa^{(1)}e\right) ^{(2)}$. \ By Lemma 
\ref{p13}, there is a strong von Neumann inverse $v\in eRf$ of\ $fa^{(1)}e$
such that $u\leq ^{\oplus}v$ and consequently, by Lemma 10, $euf\leq ^{\oplus }evf$. Thus, we have $x\leq ^{\oplus }evf$. Next, we will show that $evf\leq ^{\oplus }a$. We have $\left( evf\right)
a^{(1)}\left( evf\right) =ev\left( fa^{(1)}e\right) vf=evf$ as $v\in
(fa^{(1)}e)^{(1, 2)}$. \ Hence $\{a^{(1)}\}\subseteq \{(evf)^{(1)}\}$. \ By Lemma 
\ref{p3}, $evf\leq ^{\oplus }a$. Thus $%
\max C\subseteq \{evf:v$ is a strong von Neumann inverse of\ $fa^{(1)}e\}$. Clearly, $\max C$ is non-empty unless $evf=a$ for each choice of $v$ but this is not possible as we have assumed $a\not\in S$.

Now suppose $evf,ev^{\prime }f\in C$ such that $v,v^{\prime }$ are strong
von Neumann inverses of\ $fa^{(1)}e$ and $evf\leq ^{\oplus }ev^{\prime }f$.
\ Therefore $ev^{\prime }fR=evfR\oplus (ev^{\prime }f-evf)R$. \ Now we want
to show that $ev^{\prime }fR=evfR$. \ As $evf,ev^{\prime }f\in C$, $evf\leq
^{\oplus }a$ and $ev^{\prime }f\leq ^{\oplus }a$. \ Thus $%
\{a^{(1)}\}\subseteq \{\left( evf\right) ^{(1)}\}$ and $\{a^{(1)}\}\subseteq
\{\left( ev^{\prime }f\right) ^{(1)}\}$. \ So let $a^{(1)}$ be a common von
Neumann inverse of $evf$ and $ev^{\prime }f$. \ By assumption $evfR\subseteq
ev^{\prime }fR$. \ As shown in Lemma \ref{p14}, $\left( ev^{\prime }f\right)
a^{(1)}\left( evf\right) =evf$ and $\left( ev^{\prime }f\right)
a^{(1)}\left( ev^{\prime }f\right) =\left( ev^{\prime }f\right) $. \ Now $%
\left( ev^{\prime }f\right) R=ev^{\prime }fa^{(1)}R=ev^{\prime
}fa^{(1)}eR=ev^{\prime }(fa^{(1)}evfa^{(1)}e)R\subseteq ev^{\prime
}fa^{(1)}evfR=evfR\subseteq ev^{\prime }fR$. \ Thus $ev^{\prime }fR=evfR$. \
Similarly we can show that $Rev^{\prime }f=Revf$. \ 

As $Rev^{\prime }f=Revf$, we claim that $ev^{\prime }f=evf$. $\ $Let $%
ev^{\prime }f=revf$ for some $r\in R$. \ Now $evf=ev^{\prime
}fa^{(1)}evf=(revf)a^{(1)}evf=r(evf)=ev^{\prime }f$. \ Thus $evf=ev^{\prime
}f$. \ Hence $\max C=\{evf:v$ is a strong von Neumann inverse of\ $%
fa^{(1)}e\}$.
\end{proof}

We now provide an example to illustrate the previous theorem. \ 

\begin{example}
Note that we are choosing $f$ to be of rank two.\ \ So any maximal element
will have, at most, rank two. \ Choose $e= 
\begin{bmatrix}
1 & 0 & 0 & 0 \\ 
0 & \frac{1}{2} & \frac{1}{2} & 0 \\ 
0 & \frac{1}{2} & \frac{1}{2} & 0 \\ 
0 & 0 & 0 & 1
\end{bmatrix}
$ and $f= 
\begin{bmatrix}
\frac{1}{2} & \frac{1}{4} & 0 & 0 \\ 
1 & \frac{1}{2} & 0 & 0 \\ 
0 & 0 & 0 & 0 \\ 
0 & 0 & 0 & 1
\end{bmatrix}
$. \ Suppose $a= 
\begin{bmatrix}
1 & 0 & 0 & 0 \\ 
0 & 1 & 0 & 0 \\ 
0 & 0 & 1 & 0 \\ 
0 & 0 & 0 & 1
\end{bmatrix}
$. $\ $Then one choice for $a^{(1)}$ is $a^{(1)}= 
\begin{bmatrix}
1 & 0 & 0 & 0 \\ 
0 & 1 & 0 & 0 \\ 
0 & 0 & 1 & 0 \\ 
0 & 0 & 0 & 1
\end{bmatrix}
$ and $fa^{(1)}e= 
\begin{bmatrix}
\frac{1}{2} & \frac{1}{8} & \frac{1}{8} & 0 \\ 
1 & \frac{1}{4} & \frac{1}{4} & 0 \\ 
0 & 0 & 0 & 0 \\ 
0 & 0 & 0 & 1
\end{bmatrix}
$. \ For our choice of a strong von Neumann inverse of $fa^{(1)}e$, we first
choose its Moore-Penrose inverse and later its group inverse, as both are
also strong von Neumann inverses. \ Let $v_{1}$ be the Moore-Penrose inverse
of $fa^{(1)}e$. \ Then $v_{1}= 
\begin{bmatrix}
\frac{16}{45} & \frac{32}{45} & 0 & 0 \\ 
\frac{4}{45} & \frac{8}{45} & 0 & 0 \\ 
\frac{4}{45} & \frac{8}{45} & 0 & 0 \\ 
0 & 0 & 0 & 1
\end{bmatrix}
$ and $ev_{1}f= 
\begin{bmatrix}
\frac{8}{9} & \frac{4}{9} & 0 & 0 \\ 
\frac{2}{9} & \frac{1}{9} & 0 & 0 \\ 
\frac{2}{9} & \frac{1}{9} & 0 & 0 \\ 
0 & 0 & 0 & 1
\end{bmatrix}
$. \ Now $ev_{1}f\leq ^{-}a$ because $rank(a-ev_{1}f)=2=4-2=rank\left(
a\right) -rank(ev_{1}f)$. \ Thus $ev_{1}f\in $\bigskip $\max C$.

We now find another element of $\max C$. \ The group-inverse $v_{2}$ of $%
fa^{(1)}e$ is $v_{2}= 
\begin{bmatrix}
\frac{8}{9} & \frac{2}{9} & \frac{2}{9} & 0 \\ 
\frac{16}{9} & \frac{4}{9} & \frac{4}{9} & 0 \\ 
0 & 0 & 0 & 0 \\ 
0 & 0 & 0 & 1
\end{bmatrix}
$. \ Then $ev_{2}f= 
\begin{bmatrix}
\frac{2}{3} & \frac{1}{3} & 0 & 0 \\ 
\frac{2}{3} & \frac{1}{3} & 0 & 0 \\ 
\frac{2}{3} & \frac{1}{3} & 0 & 0 \\ 
0 & 0 & 0 & 1
\end{bmatrix}
$. \ Now $ev_{2}f\leq ^{-}a$ because $rank(a-ev_{2}f)=2=4-2=rank\left(
a\right) -rank(ev_{2}f)$. \ Thus $ev_{2}f\in $\bigskip $\max C$. \ \ 
\end{example}

\section{An Application}

In this section, as an application of our main theorem on maximal elements,
we derive the unique shorted operator $a_{S}$ of Anderson-Trapp (See \cite{A2}, Theorem 1)
that was also studied by Mitra-Puri (See \cite{M3}, Theorem 2.1). \ We
believe that there will be other such applications.

Throughout this section $R$ will denote the ring of $n\times n$ matrices
over the field of complex numbers, $\mathbb{C}$. \ For any matrix or vector $%
u$, $u^{\ast }$ will denote the \textit{conjugate transpose} of $u$. \ In
this section $S$ will denote the set of positive semidefinite matrices. \ 

Recall, the \textit{Loewner order}, $\leq _{L}$, on the set $S$ of positive
semidefinite matrices in $R$ is defined as follows: for $a,b\in S$, $a\leq
_{L}b$ if $b-a\in S$. \ 

Suppose $a\in S$ and $c\in R$. \ As in the previous section, $eR=aR\cap cR$, 
$e=e^{2},$ and choose $f=e^{\ast }$. Clearly, $\ f\in Ra$ because $a$ is
hermitian. \ Let $C_{L}=\{s\in eRf\cap S:s\leq _{L}a\}=\{s\in eSf:s\leq
_{L}a\}$.

Under this terminology, the set $C$ in the previous section will become, $%
C=\{s\in eSf:s\leq ^{\oplus }a\}$. \ 

We will assume that $a\notin eSf$. This is equivalent to the assumption that $rank(e)\neq rank(a)$, as shown
in the remark below.

\begin{remark}
\label{p19}$rank(e)=rank(a)$ if and only if $a\in eSf.$

\begin{proof}
Suppose $rank(e)=rank(a)$. \ So $eR=aR$ as $eR\subseteq aR$. \ Then $a=ex$
for some $x\in R$ and by taking conjugates, $a=x^{\ast }e^{\ast }$, $i.e.$, $%
a\in Re^{\ast }$. \ Hence, $a\in eRe^{\ast }$. \ As $a\in S$, $a\in S\cap
eRe^{\ast }=eSe^{\ast }$. \ For if $exe^{\ast }\in S$ then $exe^{\ast
}=e\left( exe^{\ast }\right) e^{\ast }\in eSe^{\ast }$ and so $S\cap
eRe^{\ast }\subseteq eSe^{\ast }$. \ The reverse inclusion is obvious.

Conversely, suppose $a\in eSf$. \ As $eR=aR\cap cR$, we have $e=ax$ and so $%
rank(e)\leq rank(a)$. \ As $a\in eSf$, $a=ese^{\ast }$ for some $s\in S$. \
Therefore $rank(a)\leq rank(e)$. \ Hence, $rank(e)=rank(a)$.
\end{proof}
\end{remark}

The following lemma is folklore.

\begin{lemma}
\label{p20}Suppose $a,b\in S$. \ If $a\leq ^{\oplus }b$ then $a\leq _{L}b.$

\begin{proof}
Suppose $a\leq ^{\oplus }b$. \ Equivalently, $\left( b-a\right) \leq
^{\oplus }b$ and by Lemma \ref{p3} we know that $\{b^{(1)}\}\subseteq
\{\left( b-a\right) ^{(1)}\}$. \ Thus, $b^{\dag }$ is a von Neumann inverse
of $\left( b-a\right) $. \ From \cite{L3}, as $b$ is positive semidefinite, $%
b^{\dag }$ is positive semidefinite. \ Thus $b-a=\left( b-a\right) b^{\dag
}\left( b-a\right) \geq _{L}0$. \ Hence $\left( b-a\right) \in S$ and $a\leq
_{L}b$.
\end{proof}
\end{lemma}

\begin{theorem}
\label{p21}Let $a\in S$ \ and let $f_{a}^{\dag }$ be the $a$-weighted
Moore-Penrose inverse of $f$. \ Then $\ \max
C=\max C_{L}=\{af_{a}^{\dag }f\}.$
\end{theorem}

\begin{proof}
By Theorem 13, $\max C=\{evf:v$ is a strong von Neumann inverse of $%
fa^{(1)}e\}$. \ By assumption, $e\in aR$ and so $e=ax$ for some $x\in R$. \
By taking conjugates, $e^{\ast }=x^{\ast }a$ as $a\in S$. \ In addition, as $%
f\in Ra$, $f=ya$ for some $y\in R$. \ This yields that $%
fa^{(1)}e=yaa^{(1)}ax=yax$ and thus $fa^{(1)}e$ is independent of the choice
of $a^{(1)}$. \ We may then choose the Moore-Penrose inverse $a^{\dag }$ for 
$a^{(1)}$. \ Next, we want to show that a strong von Neumann inverse of $%
fa^{\dag }e$ is also unique. \ Note that $fa^{\dag }e=e^{\ast }a^{\dag }e$
is positive semidefinite, as the Moore-Penrose inverse of a positive
semidefinite element is positive semidefinite \cite{L3}. \ As $a\in S$, we
can write $a=zz^{\ast }$ for some $z\in R$. \ Now $fR=yaR=yaa^{\dag
}aR=fa^{\dag }aR=fa^{\dag }R=fzz^{\ast }R=fzR=\left( fz\right) \left(
fz\right) ^{\ast }R=fzz^{\ast }f^{\ast }R=fa^{\dag }eR$. \ Similarly $%
Re=Rfa^{\dag }e$. \ It follows that $f=fa^{\dag }ep$ and $e=qfa^{\dag }e$
for some $p,q\in R$. \ Consider an element $evf\in \max C$. $\ $Then $%
evf=qfa^{\dag }evfa^{\dag }ep=qfa^{\dag }ep,$ showing that $evf$ is
independent of the choice of strong von Neumann inverse $v$ of $fa^{\dag }e$%
. \ Thus $\max C$ is a singleton set consisting of the element $e\left(
fa^{\dag }e\right) ^{\dag }f$. \ Since $a\in S$, $a^{\dag }\in S$ and hence $%
e\left( e^{\ast }a^{\dag }e\right) ^{^{\dag }}f=e\left( fa^{\dag }e\right)
^{\dag }f\in S$.

Next, we proceed to show that $\max C=$ $\{af_{a}^{\dag }f\}$ also. \ Recall
that $af_{a}^{\dag }f$ is hermitian and so $af_{a}^{\dag }f=\left(
af_{a}^{\dag }f\right) ^{\ast }=f^{\ast }\left( f_{a}^{\dag }\right) ^{\ast
}a^{\ast }=\left( f_{a}^{\dag }f\right) ^{\ast }a$. \ Since $f_{a}^{\dag }f$
is an idempotent, we get $af_{a}^{\dag }f=a(f_{a}^{\dag }f)(f_{a}^{\dag
}f)=\left( f_{a}^{\dag }f\right) ^{\ast }a(f_{a}^{\dag }f)$ and thus $%
af_{a}^{\dag }f\in S$. \ 

We now prove that $af_{a}^{\dag }f\leq ^{\oplus }a$. \ Let $a^{(1)}$ be an
arbitrary von Neumann inverse of $a$. \ Then $\left( af_{a}^{\dag }f\right)
a^{(1)}\left( af_{a}^{\dag }f\right) =(af_{a}^{\dag })(ya)a^{(1)}\left(
af_{a}^{\dag }f\right) =(af_{a}^{\dag }y)aa^{(1)}a(f_{a}^{\dag
}f)=af_{a}^{\dag }yaf_{a}^{\dag }f=af_{a}^{\dag }ff_{a}^{\dag
}f=af_{a}^{\dag }f$. \ Hence $\{a^{(1)}\}\subseteq \{\left( af_{a}^{\dag
}f\right) ^{(1)}\}.$ Consequently, by Lemma 3, $af_{a}^{\dag }f\leq ^{\oplus
}a$ which gives $af_{a}^{\dag }f\in C$.

Furthermore, by Lemma 16, $af_{a}^{\dag }f\leq ^{\oplus }a$ gives $%
af_{a}^{\dag }f\leq _{L}a$ and hence $af_{a}^{\dag }f\in C_{L}.$

Finally, we show that for every $d\in C_{L}$, $d\leq _{L}af_{a}^{\dag }f$. \
As $d\in S\subseteq Rf$, write $d=uf$ for some $u\in R$. \ Then $%
df_{a}^{\dag }f=uff_{a}^{\dag }f=uf=d=(f_{a}^{\dag }f)^{\ast }d\left(
f_{a}^{\dag }f\right) $ as $d$ is hermitian. \ Now consider $af_{a}^{\dag
}f-d=\left( f_{a}^{\dag }f\right) ^{\ast }a\left( f_{a}^{\dag }f\right)
-\left( f_{a}^{\dag }f\right) ^{\ast }d\left( f_{a}^{\dag }f\right) =\left(
f_{a}^{\dag }f\right) ^{\ast }\left( a-d\right) \left( f_{a}^{\dag }f\right)
,$ which is positive semidefinite and thus $af_{a}^{\dag }f-d\in S$. \ Hence 
$d\leq _{L}af_{a}^{\dag }f$.

Thus $af_{a}^{\dag }f$ is the unique maximal element in $C_{L}$ provided $%
af_{a}^{\dag }f\neq a$. We have shown above that $af_{a}^{\dag }f\in C_{L}$
and thus $af_{a}^{\dag }f\in eSf$. \ But by assumption $a\notin eSf$ . \ So $%
af_{a}^{\dag }f\neq a$. Therefore, $af_{a}^{\dag }f$ is unique maximal
element in $C_{L}$ and it also belongs to $C$ as we have already proven that 
$af_{a}^{\dag }f\leq ^{\oplus }a$.

Now, because $e\left( fa^{\dag }e\right) ^{\dag }f$ is the unique maximal
element in $C\ $and $af_{a}^{\dag }f\in C$, $af_{a}^{\dag }f$ $\leq ^{\oplus
}e\left( fa^{\dag }e\right) ^{\dag }f$ . \ By Lemma 16, $af_{a}^{\dag }f\leq
_{L}e\left( fa^{\dag }e\right) ^{\dag }f$ as $e\left( fa^{\dag }e\right)
^{\dag }f\in C_{L}$. \ We have shown above that for every element $d\in C_{L}
$, $d\leq _{L}$ $af_{a}^{\dag }f$ and thus $af_{a}^{\dag }f$ $=$ $e\left(
fa^{\dag }e\right) ^{\dag }f$.\ \ Hence, $\max C=\max C_{L}=\{af_{a}^{\dag
}f\}$ as desired.
\end{proof}

\bigskip

The following examples demonstrate the result proved in the previous
theorem, i.e. $af_{a}^{\dag }f$ $=$ $e\left( fa^{\dag }e\right) ^{\dag }f$
and so $\max C=\max C_{L}=\{af_{a}^{\dag }f\}$. \ Furthermore, $\max C$
agrees with the formula given by Anderson-Trapp for computing the shorted
operator $a_{S}$ when we are given the impedance matrix $a.$

\ The Anderson-Trapp formula states that if $a$ is the $n\times n$ impedance
matrix then the shorted operator of $a$ with respect to the $k$-dimensional
subspace $S$ \ (shorting $n-k$ ports) is given by $a_{S}= 
\begin{bmatrix}
a_{11}-a_{12}a_{22}^{\dag }a_{21} & 0 \\ 
0 & 0
\end{bmatrix}
$, where $a$ is partitioned as $a= 
\begin{bmatrix}
a_{11} & a_{12} \\ 
a_{21} & a_{22}
\end{bmatrix}
$ such that $a_{11}$ is a $k\times k$ matrix. \ We show that the maximum
element $af_{a}^{\dag }f$ obtained by us\ is permutation equivalent to $%
a_{S} $, i.e. $P^{T}af_{a}^{\dag }f$ $P=a_{S}$ for some permutation matrix $%
P.$

\bigskip

\begin{example}
Let $e= 
\begin{bmatrix}
\frac{1}{2} & 0 & \frac{1}{2} & 0 \\ 
0 & 0 & 0 & 0 \\ 
\frac{1}{2} & 0 & \frac{1}{2} & 0 \\ 
0 & 0 & 0 & 1
\end{bmatrix}
$ and then $f=e^{\ast }= 
\begin{bmatrix}
\frac{1}{2} & 0 & \frac{1}{2} & 0 \\ 
0 & 0 & 0 & 0 \\ 
\frac{1}{2} & 0 & \frac{1}{2} & 0 \\ 
0 & 0 & 0 & 1
\end{bmatrix}
$. \ Suppose $a= 
\begin{bmatrix}
1 & 0 & 1 & 0 \\ 
0 & 1 & 0 & 0 \\ 
1 & 0 & 1 & 0 \\ 
0 & 0 & 0 & 1
\end{bmatrix}
$. $\ $\ Then one may check that $f_{a}^{\dag }=f= 
\begin{bmatrix}
\frac{1}{2} & 0 & \frac{1}{2} & 0 \\ 
0 & 0 & 0 & 0 \\ 
\frac{1}{2} & 0 & \frac{1}{2} & 0 \\ 
0 & 0 & 0 & 1
\end{bmatrix}
$. \ 

\ So $af_{a}^{\dag }f$ $= 
\begin{bmatrix}
1 & 0 & 1 & 0 \\ 
0 & 0 & 0 & 0 \\ 
1 & 0 & 1 & 0 \\ 
0 & 0 & 0 & 1
\end{bmatrix}
$.\bigskip\ \ We now show that $af_{a}^{\dag }f=$ $e\left( fa^{\dag
}e\right) ^{\dag }f.$ Now, $a^{\dag }= 
\begin{bmatrix}
\frac{1}{4} & 0 & \frac{1}{4} & 0 \\ 
0 & 1 & 0 & 0 \\ 
\frac{1}{4} & 0 & \frac{1}{4} & 0 \\ 
0 & 0 & 0 & 1
\end{bmatrix}
$ and $\left( fa^{\dag }e\right) ^{\dag }= 
\begin{bmatrix}
1 & 0 & 1 & 0 \\ 
0 & 0 & 0 & 0 \\ 
1 & 0 & 1 & 0 \\ 
0 & 0 & 0 & 1
\end{bmatrix}
$. \ Thus $e\left( fa^{\dag }e\right) ^{\dag }f= 
\begin{bmatrix}
1 & 0 & 1 & 0 \\ 
0 & 0 & 0 & 0 \\ 
1 & 0 & 1 & 0 \\ 
0 & 0 & 0 & 1
\end{bmatrix}
$. \ Hence $e\left( fa^{\dag }e\right) ^{\dag }f= 
\begin{bmatrix}
1 & 0 & 1 & 0 \\ 
0 & 0 & 0 & 0 \\ 
1 & 0 & 1 & 0 \\ 
0 & 0 & 0 & 1
\end{bmatrix}
=af_{a}^{\dag }f$ as proved in the theorem. \ We may verify that $%
af_{a}^{\dag }f\leq ^{\oplus }a$. \ This follows from $rank(a)-rank(af_{a}^{%
\dag }f)=3-2=1=rank(a-af_{a}^{\dag }f)$. \ We know then $af_{a}^{\dag }f\leq
_{L}a$. \ Thus $\max C=\max C_{L}=\{af_{a}^{\dag }f\}$. \ 

We now compute the shorted operator as given by Anderson-Trapp. \ We
partition $a$ as follows: \ $a= 
\begin{bmatrix}
\begin{bmatrix}
1 & 0 & 1 \\ 
0 & 1 & 0 \\ 
1 & 0 & 1
\end{bmatrix}
& 
\begin{bmatrix}
0 \\ 
0 \\ 
0
\end{bmatrix}
\\ 
\begin{bmatrix}
0 & 0 & 0
\end{bmatrix}
& 
\begin{bmatrix}
1
\end{bmatrix}
\end{bmatrix}
$. \ 

Then $a_{S}= 
\begin{bmatrix}
\begin{bmatrix}
1 & 0 & 1 \\ 
0 & 1 & 0 \\ 
1 & 0 & 1
\end{bmatrix}
- 
\begin{bmatrix}
0 \\ 
0 \\ 
0
\end{bmatrix}
\begin{bmatrix}
1
\end{bmatrix}
^{\dag } 
\begin{bmatrix}
0 & 0 & 0
\end{bmatrix}
& 0 \\ 
0 & 0
\end{bmatrix}
= 
\begin{bmatrix}
1 & 0 & 1 & 0 \\ 
0 & 1 & 0 & 0 \\ 
1 & 0 & 1 & 0 \\ 
0 & 0 & 0 & 0
\end{bmatrix}
$. \ Now for $P= 
\begin{bmatrix}
1 & 0 & 0 & 0 \\ 
0 & 0 & 0 & 1 \\ 
0 & 0 & 1 & 0 \\ 
0 & 1 & 0 & 0
\end{bmatrix}
$, $Paf_{a}^{\dag }fP^{T}=a_{S}$.
\end{example}

\bigskip

\bigskip

\end{document}